\documentclass[12pt, a4paper]{article}
\usepackage{enumerate,amsthm,amsfonts, mathrsfs, amssymb, latexsym, dsfont}
\usepackage{times}
\usepackage{mathptmx}
\usepackage{amsmath}
\usepackage{avant}
\usepackage{verbatim}
\usepackage{graphicx}
\usepackage[left=2.7cm,right=2.7cm,top=4.5cm,bottom=3.5cm,headsep=2cm]{geometry}
\usepackage{authblk}
\usepackage[round]{natbib}
\author[a]{Carla~Mereu}
\author[b]{Robert~Stelzer}

\affil[a]{\small{Ulm University,  Institute of Mathematical Finance, Helmholtzstrasse 18, 89081 Ulm, Germany. ccarlammereu@gmail.com}}
\affil[b]{\small{Ulm University, Institute of Mathematical Finance, Helmholtzstrasse 18, 89081 Ulm, Germany. robert.stelzer@uni-ulm.de}}

\title{A BSDE arising in an exponential utility maximization problem in a pure  jump market model}

\allowdisplaybreaks
\renewcommand{\tilde}{\widetilde}

\newcommand{\xproc}[1]{\ensuremath{(#1_{t})_{t \in [0,T]}}}
\newcommand{\R}{\ensuremath{\mathbb{R}}}

\newcommand{\E}{\ensuremath{\mathbb{E}}}

\newcommand{\spzpf}{$(\Omega, \mathcal{F}, \ensuremath{(\mathcal{F}_t)_{t \in [0,T]}}, {\ensuremath{\mathbb{P}}})$ }
\renewcommand{\P}{\ensuremath{\mathbb{P}}}
\newcommand{\F}{\ensuremath{\mathcal{F}}}

\newcommand{\ud}{\,\mathrm{d}}

\DeclareMathOperator*{\argmin}{argmin}
\newcommand{\systeq}[4]{\begin{equation}
\left\{
\begin{array}{rl}
#1=& #2,\qquad \forall t\in [ 0,T ]\\
#3=& #4
\end{array}
\right.
\end{equation} }
\newcommand{\matlabcode}[3]%
{
\definecolor{number}{gray}{0.6}
\definecolor{keywords}{rgb}{1.0,0.3,0.3}
\definecolor{comments}{rgb}{0.1,0.65,0.1}
\definecolor{strings}{rgb}{0.3,0.0,1.0}
\lstset{language=Matlab,
        morekeywords={switch,case},
        sensitive=true,
        showspaces=false, 
        basicstyle=\ttfamily\small\mdseries,
      	keywordstyle=\bfseries\color{keywords},
 	      commentstyle=\color{comments},
 	      stringstyle=\color{strings},
 	      numbers=left,
 	      numberstyle=\scriptsize\color{number},
 	      stepnumber=1,
 	      breaklines=true,
 	      frame=none,
        showstringspaces=false,
 	      tabsize=4,
 	      xleftmargin=\bigskipamount,
 	      xrightmargin=\bigskipamount,
 	      aboveskip=\bigskipamount,
      	belowskip=0pt
}
\lstinputlisting[caption=#2, label=#3, frame=tb]{#1}
}
\renewcommand{\[}{\left[}
\renewcommand{\]}{\right]}
\renewcommand{\(}{\left(}
\renewcommand{\)}{\right)}
\renewcommand{\leq}{\leqslant}
\renewcommand{\geq}{\geqslant}

\DeclareMathOperator{\tr}{tr}

\numberwithin{equation}{section}
\newtheorem{Teo}{Theorem}[section]
\newtheorem{Lem}[Teo]{Lemma}
\newtheorem{Prop}[Teo]{Proposition}

\theoremstyle{definition}
\newtheorem{Def}[Teo]{Definition}
 
\theoremstyle{remark}
\newtheorem{oss}[Teo]{Remark}

\date{}
     
\begin{document}

\maketitle
\begin{center}
\emph{Dedicated to Bernt \O ksendal on the occasion of his 70th birthday}
\end{center}

\medskip
\begin{abstract}
We consider the problem of utility maximization with exponential preferences in a market where the traded stock/risky asset price is modelled as a L\'evy-driven pure jump process (i.e. the driving L\'evy process has no Brownian component). In this setting, we study the terminal utility optimization problem in the presence of a European contingent claim.
We consider in detail the BSDE (backward stochastic differential equation)  characterising the value function when using an exponential utility function. First we analyse the well-definedness of the generator. This leads to some conditions on the market model related to conditions for the market to admit no free lunches. Then we give bounds on the candidate optimal strategy. 

Thereafter, we discuss the example of a cross-hedging problem and, under severe assumptions on the structure of the claim, we give explicit solutions. Finally, we establish an explicit solution for a related BSDE with a suitable terminal condition but a simpler generator.
\end{abstract}


\vspace{0.3cm}
\vspace{0.3cm}

\section{Introduction}\label{sec:Introduction}
In the context of utility maximization, exponential utility has been widely used because of its nice analytic tractability. In particular, it shows fundamental separation properties when dealing with contingent claims.

For It\^o-diffusion and continuous martingale models, backward stochastic differential equation (BSDE for short) methods have been applied in order to relax the assumption of convexity on the constraint set in the seminal paper by \cite{hu2005utility} as well as in many following papers by various authors. \cite{hu2005utility} rely on the so-called ``martingale optimality principle'' to derive a BSDE characterizing the solution of the problem. 
In the presence of jumps, the first paper using this methodology in the context of utility maximization, at least to the best of our knowledge, is \cite{becherer2006bounded} (we refer the reader to e.g. \cite{OeksendalSulem2007} for a general introduction into optimal control with jump processes).
 \cite{becherer2006bounded} considers again an It\^o-diffusion market model, but relaxes the assumptions on the filtration. This is assumed to be the natural filtration generated by a multidimensional Brownian motion and an independent integer-valued random measure. \cite{morlais2009utility,morlais2010new} extends the results to the case of a L\'evy-It\^o diffusion model, i.e. she allows also jumps in the stock price process. In all the above mentioned papers, a fundamental assumption is that the Gaussian covariance matrix is strictly positive-definite, which ensures the existence of an equivalent martingale measure.
In the absence of a Gaussian component in the dynamics of the price, additional conditions on the L\'evy measure and the drift term need to be imposed for the model to admit an equivalent (local) martingale measure  (see e.g. \cite{bardhan1996martingale, protter2008no, kardaras2009no}).

The purpose of this paper is to analyze the ``complementary'' case to the one studied in \cite{becherer2006bounded}, namely when the stock price is a L\'evy-driven pure jump process and the filtration is generated by its associated jump measure and an independent Brownian motion. In this setting, we first construct a BSDE by means of the martingale optimality principle in the standard way (i.e. as e.g. in \cite{becherer2006bounded} or \cite{morlais2009utility,morlais2010new}) and give conditions for the corresponding generator to be well-defined.  Let us also mention here that these conditions turn out to be identical to conditions known to imply no free lunch with vanishing risk in several situations where such NFLVR conditions are known. After deriving the BSDE via the martingale optimality principle, we follow a different route than the previously cited papers, which show existence and uniqueness of solutions in appropriate spaces. Instead we first simply assume that we have a solution to the BSDE and study whether this allows us to obtain a solution to the utility optimization problem including an optimal strategy. Later on we, similarly to  \cite{richter2012explicit},  consider a concrete problem were we can directly find a solution to the BSDE. The reason is that we want to obtain explicit solutions and that we want to consider unbounded constraint sets for the strategy for which it seems to be very hard to prove optimality of strategies in the presence of jumps in general (note  that \cite{morlais2009utility,morlais2010new} shows that in her jump diffusion market the optimization problem has a solution given in terms of the BSDE, but nothing is said about the existence of an optimal strategy for unbounded constraint sets).

A motivation for the considered problem is the following application. We are interested in investigating a cross-hedging problem in the case where the stock price is described by a pure jump process but the investor wants to hedge a derivative on another (illiquid) asset. Imagine that the price of this asset is -- as in some cross-hedging problems -- strongly correlated with the price of the traded stock, but that their logspread is not constant and exhibits a mean reverting behaviour. This is in the present paper modeled by a market where the stock price is a L\'evy-driven pure jump process and the logspread follows an Ornstein-Uhlenbeck process.

The remainder of the paper is structured as follows. In Section \ref{sec:Model} we introduce the model. Section \ref{sec:WellPosedness} deals with the well-posedness of the optimization problem and, in particular, Theorem \ref{nflvr} gives conditions on the market parameters such that the problem is well-defined. Moreover, we give bounds on the ``candidate optimal strategy'' and conditions when it is indeed optimal. In Section \ref{sec:Cross hedging} we illustrate an example of a cross-hedging problem where explicit solutions can be obtained, under the assumption that the claim is logarithmic in the price of the illiquid asset. Finally, in Section \ref{sec:exp terminal} we discuss the difficulties of extending the approach of Section \ref{sec:Cross hedging} to more general claims.

\section{The market model}\label{sec:Model}
We assume given a filtered probability space \spzpf with $T>0$ a finite time horizon and a filtration \xproc{\F} satisfying the usual conditions. Assume that the above filtration is generated by the following two processes, independent of each other:
\begin{itemize}
\item a standard (one-dimensional) Brownian motion \xproc{W};
\item a real-valued Poisson point process $p$ with associated counting measure $N_p(\ud t,\ud x),$ and compensator $\widehat{N}_p(\ud t,\ud x)=\nu(\ud x)\ud t,$ where the L\'evy measure $\nu$ is positive and satisfies
\begin{equation}
\nu(\{0\}) = 0,\ \text{and}\ \int_{\R^*}(1\wedge|x|^2)\nu(\ud x)<\infty.
\end{equation}
Let $\tilde{N}_p$ denote its compensated counting measure.
\end{itemize}
Here, we denote $\R\setminus\{0\}$ by $\R^*$.
\xproc{\F} is hence the right-continuous filtration generated by the two processes, and completed by the $\P$-null sets. We denote by $\mathcal{P}=\mathcal{P}(\mathcal{F}_t)$ the associated predictable $\sigma$-algebra on $[0,T]\times\Omega$ and
 we define the following spaces:\\
$\mathrm L^2(W):=\left\{\xproc{Z} \text{ \small predictable s.t. }\E\left[\int_0^T |Z_s|^2\ud s \right]<\infty\right\},$\\
$\mathrm L^2(\tilde{N}_p) :=\left\{\xproc{U}\ \mathcal{P}\otimes\mathcal{B}(\R^*)\text{\small-measurable s.t. }\E\left[\int_{[0,T]\times \R^*} |U_s(x)|^2\nu(\ud x)\ud s \right]<\infty\right\}.$\\
For a measure $\nu$ on $\R\setminus\{0\}$ we define
$\mathrm L^0(\nu)$
as the space of all $u:\R\to\R$ measurable equipped with the (local) topology of the convergence in measure (see e.g. \cite{bauer2001measure}, \textsection 20, Part II) and we further set
$\mathrm L^2(\nu) :=\{u\in \mathrm L^0(\nu) \text{ such that } \int_{\R^*}|u(x)|^2\nu(\ud x)<\infty\}$,\\
$\mathrm L^\infty(\nu) :=\{u\in \mathrm L^0(\nu) \text{ such that }u\text{ takes bounded
  values }\nu\text{-almost surely}\}$.\\
Consider a market model consisting of a riskless asset, taken as numeraire, and a risky asset whose discounted price process $S = \xproc{S}$ evolves according to the following SDE: \systeq{\ud S_t}{S_{t-}\left(\varphi_t \ud t + \int_{\R^*}\psi_t(x)\tilde N_p(\ud t, \ud x)\right)}{S_0}{s\in(0,\infty),}
for $\varphi,\ \psi$ uniformly bounded predictable processes with $\psi\in\mathrm L^2(\tilde{N}_p)$, and $\psi>-1$ $\P-$a.s. at every time. The latter assumption ensures that the price stays strictly positive. Obviously for constant deterministic $\varphi$ and $\psi$ this is a standard exponential L\'evy model (see e.g. \cite{tankov2004financial}) which is popular in finance. However, our model is much more general. It allows not only time-inhomogeneous L\'evy models but also, for instance, models of a stochastic volatility type, as the coefficients may be stochastic. Noting that $\psi$ may depend on $W$ one could e.g. have dynamics like $ \ud S_t=S_{t-}\sqrt{\sigma_{t-}^2}dL_t$ with $L$ a pure-jump L\'evy process and $\sigma_{t-}^2$ being a square root diffusion (truncated to ensure the assumed boundedness) similar to the Heston model. So our modelling set-up allows a lot of flexibility to cover many of the stylized facts (cf. \cite{Cont2001}, \cite{Guillaumeetal1997}) of financial data sets. Finally it should be noted that we consider the price processes under the real world measure, not a risk-neutral one.

Assume now that we want to hedge a position at the terminal time $T$, i.e. we know that we will have to pay an $\F_T$-measurable discounted amount $B$. 
We want to maximize the expected utility of the discounted terminal wealth. The discounted wealth process $X^{x,\pi}$ is composed of the initial capital $x\in\R,$ and gains from trading with a self-financing strategy $\pi$ in the market. The strategy $\pi$ corresponds to the discounted amount of money invested in the stock, the number of shares is $\pi_t/S_t$.

As we only consider discounted quantities, we will from now on often omit the adjective ``discounted''. Note also that if a riskless bank account having a zero interest rate is the numeraire then our approach considers undiscounted quantities.

The wealth process for an initial capital $x$ at time $t$ solves the equation
\begin{equation*}
X^{\pi,t,x}_s = x + \int_t^s\pi_r\frac{\ud S_r}{S_{r-}}, \quad\forall s \in[t,T]
\end{equation*}
and the dynamics of the wealth process can be rewritten
as
\begin{equation}\label{wealthdyn}
\left\{
\begin{array}{rl}
\ud X^{\pi,t,x}=& \pi_s\varphi_s\ud s+\int_{\R^*}\pi_s\psi_s(x)\tilde{N}_p(\ud s,\ud x),\qquad \forall s\in [t,T ]\\
X^{\pi,t,x}_t=& x.
\end{array}
\right.
\end{equation} 
To ease notation, we will sometimes omit one or more superscripts in the wealth process, the parameters being implicitely fixed. If not specified, the initial time is assumed to be $t=0$.\\  
We want to solve the following problem
\begin{equation*}
V(x) = \sup_{\pi\in\mathcal{A}}\mathbb{E}[U(X_T^{\pi,0,x} -B)], \quad x\in\R
\end{equation*}
where $U(x) = -\exp(-\alpha x)$ is the exponential utility function, $\alpha\in(0,\infty)$ the risk aversion parameter, and $\mathcal{A}$ is a
fixed set of admissible trading strategies defined in Definition \ref{admissible set} below.
Throughout this paper we consider only the exponential utility function (similarly to e.g. the works of \cite{becherer2006bounded} and \cite{morlais2009utility,morlais2010new}). One reason is that the exponential function has particularly nice (separation) features for our further analysis. However, the approach to solve the stochastic  optimization problem by BSDEs is in principle also applicable for different utility functions, in particular, power utilities (see e.g. \cite{hu2005utility}, \cite{richter2012explicit} for examples). However, a different utility function already results in a different generator of the BSDE and a very particular feature of the exponential utility function is that the initial wealth factors out of the value function $V$ and  that therefore the optimal strategy does not depend on the initial wealth. Hence, all the upcoming investigations have to be done anew for a different utility function and it is not clear to which extent one can e.g. find special cases allowing for explicit solutions as we will do later on. Investigating the use of different utility functions in our set-up is thus a very interesting question for future research but beyond the scope of the present paper.
\begin{Def}\label{admissible set}Let $C$ be a closed set in $\R$ with $0\in C$. The set of \emph{admissible strategies} $\mathcal{A}$ consists of all predictable processes $\pi=(\pi_t)_{0\leq t\leq T},$ $\pi$ taking values in $C$ $\lambda\otimes\P-$a.e., where $\lambda$ denotes the Lebesgue measure on $\R$, such that $\int_0^T|\pi_s\varphi_s|ds\in L^2(\Omega,P)$ and $\pi\psi\in L^2(\widetilde N_p)$ and such that the set \begin{equation}\label{admissible}\left\{\exp\left\{-\alpha X^\pi_\tau\right\}\text{ s.t. }\tau\text{ is a stopping time with values in [0,T]}\right\}\end{equation} is a uniformly integrable family.
\end{Def}
Above $\lambda$ denotes the Lebesgue measure. Regarding the integrability properties, elementary arguments including Jensen's inequality and the boundedness of  $\varphi$ show the following.
\begin{Lem}\label{intlem}
 Assume that $\pi$ is predictable and $E\left(\int_0^T|\pi_s|^2ds\right)<\infty$. 

If $|\psi_s|_{L^2(\nu)}:=\left(\int_{\R^*}\psi^2_s(x)\nu(dx)\right)^{1/2}$ is bounded on $\Omega\times[0,T]$, then   $\int_0^T|\pi_s\varphi_s|ds\in L^2(\Omega,P)$ and $\pi\psi\in L^2(\widetilde N_p)$.
\end{Lem}

We define a dynamic version of the value function associated to the problem as follows
\begin{equation}\label{value function}
V_t(x) = \sup_{\pi\in\mathcal{A}}\E[U(X_T^{x,t,\pi} -B)], \quad x\in\R, t\in[0,T].
\end{equation}
and we will now describe the solution to this problem by a BSDE of the type:
\systeq{\label{bsde}-\ud Y_t}{f(t,Y_{t-},Z_t,U_t)\ud t -Z_t\ud W_t - \int_{\R^*} U_t(x)\tilde N(\ud t,\ud x)}{Y_T}{B.}
A BSDE is determined by a terminal condition (in this case the claim $B$) and a generator. The following derivation  via the ``martingale optimality principle'' follows the standard route (see e.g. \cite{hu2005utility}, \cite{morlais2009utility} ) and so we only sketch it:\\
The martingale optimality principle (see \cite{rogerswilliams2000} , for instance) implies that we should decompose the process $U(X^{x,\pi}_t - Y_t)$ in such a way that it is a supermartingale for every admissible $\pi$ and a martingale for some admissible strategy $\pi^*$.\\
We start by applying It\^o's formula to the function $U$ composed with the process $X^{x,\pi}-Y$ and recalling \eqref{wealthdyn}. We obtain 
\begin{align*}
\ud U(X-Y)_t =U(X_{t-} - Y_{t-})&\left\{\alpha Z_t \ud W_t +\int_{\R^*} (e^{-\alpha(\pi_s\psi_s(x)-U_s(x))}-1)\tilde N(\ud t, \ud x) \right.\\
&- \alpha f(t,Y_{t-},Z_t,U_t)\ud t  -\alpha \pi_t\varphi_t\ud t + \frac12 \alpha^2|Z_t|^2\ud t  \\
&\left.+ \int_{\R^+} \left(e^{-\alpha(\pi_t \psi_t(x) - U_t(x))} - 1 +\alpha(\pi_t\psi_t(x) - U_t(x))\right)\nu(\ud x) \ud t\right\}.
\end{align*}
We want to choose the generator $f$ in such a way that the process above is a supermartingale for every admissible strategy. We hence look at the finite variation part which can be seen to be of the form $-e^{A^{\pi}_t}$ where 
\begin{equation*}
A^\pi_t = \int_0^t \left[\left(\frac12 \alpha^2 Z_s^2 - \alpha \pi_s\varphi_s -\alpha f(s,Y_{s-},Z_s,U_s) \right) + \int_{\R^*} \alpha g_{\alpha}(U_s(x)-\pi_s\psi_s(x)) \nu(\ud x)\right]\ud s
\end{equation*}
and $g_{\alpha}$ is the real convex function defined by $$g_\alpha(y) = \frac{e^{\alpha y}-\alpha y -1}{\alpha}.$$
In particular, the required supermartingale property is satisfied if the argument of the integral defining $A^{\pi}$ is non-negative. The martingale optimality principle therefore implies the following choice of the generator $f$:
\begin{align}
f(t,y,z,u) =f(t,z,u) := &\inf_{\pi\in C} \left\{ \frac\alpha2 |z|^2 + \int_{\R^*}g_\alpha(u(x) - \pi\psi_s(x))\nu(\ud x) - \pi\varphi_s\right\}\nonumber\\
=&\inf_{\pi\in C}\left\{\int_{\R^*} g_\alpha(u(x) - \pi\psi_s(x)) \nu(\ud x) -\pi\varphi_s\right\} + \frac\alpha2 |z|^2. \label{generator}
\end{align}

Sometimes we use the notation $|u|_\alpha:=  \int_{\R^*} g_\alpha(u(x)) \nu(\ud x)$.

\section{Well-posedness of the optimization problem}\label{sec:WellPosedness}
At a first glance, it is immediate to ask whether the generator in Equation \eqref{generator} is well-defined, since we are taking the infimum of a function with a  negative linear term. 
Thus we give first conditions for the minimization problem to be well-posed.\\

\begin{Teo}\label{nflvr}
Let $\mathscr T^+$ be the set of $(t,\omega)\in[0,T]\times\Omega$ such that $\nu(\{\psi_t<0\})=0$ (i.e. the jump sizes are non-negative). Similarly, let $\mathscr T^-$ be the set of $(t,\omega)\in[0,T]\times\Omega$ such that $\nu(\{\psi_t>0\})=0$ (i.e. the jump sizes are non-positive). Assume that
\begin{itemize}
\item $\varphi_t<\int_{\R^*}\psi_t(x)\nu(\ud x) \quad \forall (t,\omega)\in\mathscr T^+,$ 
\item $\varphi_t>\int_{\R^*}\psi_t(x)\nu(\ud x)  \quad \forall (t,\omega)\in\mathscr T^-.$ 
\end{itemize}
Then for every $u\in\mathrm L^2(\nu)\cap\mathrm L^\infty(\nu)$ and $t\in[0,T]$ the function $\lambda:C\to\R$ defined as
\begin{equation*}
\lambda(\pi) = \int_{\R^*}g_\alpha(u(x)-\pi\psi_t(x))\nu(\ud x) -\pi \varphi_t
\end{equation*}
admits a minimum (in $C$).

If $C$ is equal to $\R$, the minimum is unique.
\end{Teo}
For the proof 
we  need a lemma guaranteeing that differentiation under the integral sign is allowed.
\begin{Lem}\label{differentiation under integral}
The function $\pi\mapsto \int_{\R^*}g_\alpha(u(x) - \pi\psi(x))\nu(\ud x)$ is holomorphic for every $u,\psi\in \mathrm{L}^2(\nu)\cap \mathrm{L}^\infty(\nu)$ $\mathcal B(\R^*)-$measurable and the derivative is given by
$
 \pi\mapsto \int_{\R^*}\frac{\ud}{\ud \pi}g_\alpha(u(x) - \pi\psi(x))\nu(\ud x)
$
 .
\end{Lem}
\begin{proof}
The result follows from results on differentiability under the integral sign (see e.g. \cite{mattner2001complex}), thanks to the following properties:
\begin{enumerate}[1)]
\item $g_\alpha(u(\cdot) - \pi\psi(\cdot))$ is $\mathcal B(\R^*)-$measurable for every $\pi\in\R$.\\
This simply follows from the fact that both $u$ and $\psi$ are mesaurable functions and that affine transformation as well as the function $g_\alpha$ are measurable operations.
\item $\pi\mapsto g_\alpha(u(x) - \pi\psi(x))$ is holomorphic for every $x\in\R^*$.\\
This follows from the fact that this function is the composition of the holomorphic function $g_\alpha$ with an affine function of the argument.
\item $\int_{\R^*}|g_\alpha(u(x) - \pi\psi(x))|\nu(\ud x)$ is locally bounded in $\pi$.\\
Indeed, the function $\pi\mapsto g_\alpha(u(x) - \pi\psi(x))$ is non-negative, continuous and concave and achieves its maximum at the boundary in every compact set. Moreover this maximum is integrable, since $u,\psi\in \mathrm{L}^2(\nu)\cap \mathrm{L}^\infty(\nu)$ and the function $g_\alpha$ is quadratic around $0$, which yields that $g_\alpha(u(x) - \pi\psi(x))\in\mathrm L^1(\nu)$.
\end{enumerate}
\end{proof}
\begin{proof}
Fix $t\in[0,T].$\\
Lemma \ref{differentiation under integral} implies that $\lambda$ is a holomorphic function of $\pi$ and moreover allows us to differentiate under the integral sign. Differentiating, we get 
\begin{equation*}
\lambda'(\pi) = \int_{\R^*} \psi_t(x)\left(1-e^{-\alpha (\pi\psi_t(x)-u(x))}\right)\nu(\ud x) -\varphi_t.
\end{equation*}
Lemma \ref{differentiation under integral} again guarantees that $\lambda'$ is differentiable, $\lambda''$ is continuous and it is given by
\begin{equation*}
\lambda''(\pi) = \int_{\R^*} \alpha (\psi_t(x))^2e^{-\alpha(\pi\psi_t(x) - u(x))}\nu(\ud x)\geq 0.
\end{equation*}
This implies that $\lambda'$ is increasing and that $\lambda$ is a convex function. Thanks to convexity, we have that $\lambda$ admits a minimum if $\lambda(\pi)\to+\infty$ for $\pi\to\pm \infty$.\\\\
If we show that $\lambda'$ has a zero and $\lambda'\neq 0$ for $\|\pi\|$ big enough, we can conclude, since then:
\begin{itemize}
\item $\lambda'$ is bounded from below and strictly positive for values of $\pi$ big enough, which implies that $\lambda\uparrow + \infty$ for $\pi\to+\infty$
\item $\lambda'$ is bounded from above and negative for values of $\pi$ small enough, which implies that $\lambda\uparrow+\infty$ for $\pi\to-\infty$.
\end{itemize}
We now distinguish two cases:\\
\noindent {\scshape Case 1}: Assume that $\nu(\{\psi_t>0\})>0$ and  $\nu(\{\psi_t<0\})>0$.\\
Then for $\pi\leq 0$
\begin{align*}
\lambda''(\pi) &\geq \int_{\{\psi_t>0\}}\alpha(\psi_t(x))^2e^{-\alpha(\pi\psi_t(x) - u(x))}\nu(\ud x)\\
&\geq  \int_{\{\psi_t>0\}}\alpha(\psi_t(x))^2e^{\alpha u(x)}\nu(\ud x)>0.
\end{align*}
and for $\pi\geq 0$
\begin{align*}
\lambda''(\pi) &\geq \int_{\{\psi_t<0\}}\alpha(\psi_t(x))^2e^{-\alpha(\pi\psi_t(x) - u(x))}\nu(\ud x)\\
&\geq  \int_{\{\psi_t<0\}}\alpha(\psi_t(x))^2e^{\alpha u(x)}\nu(\ud x)>0.
\end{align*}
It follows that $\lambda''$ is everywhere bounded from below by a positive number, and hence $\lim_{\pi\to+\infty}\lambda'(\pi)=+\infty$ and $\lim_{\pi\to-\infty}\lambda'(\pi)=-\infty$, and $\lambda'$ hits the origin.\\
\noindent {\scshape Case 2}: Assume that either $\nu(\{\psi_t>0\})=0$ or  $\nu(\{\psi_t<0\})=0$.\\
By symmetry, it is enough to consider $\nu(\{\psi_t>0\}) >0$ and $\nu(\{\psi_t<0\})=0$.\\
In this case it still holds that $\lim_{\pi\to-\infty}\lambda'(\pi)=-\infty$, but as $\pi\to+\infty$ its limit is not necessarily $+\infty$.\\
If $\int_{\{\psi_t>0\}}\psi_t(x) \nu(\ud x) = +\infty$,
\begin{align*}
\lim_{\pi\to+\infty}\lambda'(\pi) &= \lim_{\pi\to+\infty} \int_{\{\psi_t>0\}}\psi_t(x)\left(1-e^{-\alpha(\pi\psi_t(x) - u(x))}\right)\nu(\ud x) - \varphi_t\\
&= \int_{\{\psi_t>0\}}\psi_t(x)\nu(\ud x) -\varphi_t= +\infty,
\end{align*}
where in the first equality we used Lemma \ref{differentiation under integral}.\\
If $\int_{\{\psi_t>0\}}\psi_t(x) \nu(\ud x) < +\infty$, under the assumption $\varphi_t <\int_{\{\psi_t>0\}}\psi_t(x) \nu(\ud x)$, it holds that $\lambda'$ is again continuous and increasing and
\begin{equation*}
\lim_{\pi\to+\infty}\lambda'(\pi) = \int_{\{\psi_t>0\}}\psi_t(x) \nu(\ud x) -\varphi_t > 0.
\end{equation*}
Therefore by the mean value theorem we can conclude.

The above analysis of the second derivative shows that $\lambda$ is strictly convex and therefore for $C=\R$ the minimum is unique.
\end{proof}
As a corollary of the previous result, we obtain the well-posedness of the generator in \eqref{generator}.

\begin{oss}
\begin{enumerate}
 \item The sufficient conditions of Theorem \ref{nflvr} ensuring the well-posedness of our optimization problem are intrisically related to the market satisfying the no free lunch with vanishing risk (NFLVR) condition.  Indeed \cite{bardhan1996martingale} prove  under the assumption that $N$ has finite activity (cf. Theorem 5.2 in \cite{bardhan1996martingale}) that our conditions in Theorem \ref{nflvr} imply NFLVR. \cite{tankov2004financial} prove for the case of exponential L\'evy models (cf. Proposition 9.9 in \cite{tankov2004financial}) that these conditions are even necessary and sufficient for NFLVR. In some sense this link is not surprising, as the optimization problem should not be well-posed in the presence of arbitrages.
\item Most exponential L\'evy market models, e.g. NIG, variance Gamma or CGMY, have both positive and negative jumps which ensures that they satisfy the NFLVR conditions of \cite{tankov2004financial}. So in principle for these popular models the well-definedness of our generator is not a problem with the exception that none of these models has bounded jumps which, however, is an assumption needed at various occasions for our BSDE approach. Hence, the conditions of Theorem \ref{nflvr} are satisfied when one considers these standard model but truncates the L\'evy measure at a very high level which from the practical point of view should be a rather innocent modification. 
\end{enumerate}

\end{oss}

Note that when $U_s$, $\psi_s$ and $\varphi_s$ satisfy the conditions of Theorem \ref{nflvr} and are predictable, then a predictable $\pi^*$ such that $\pi_s^*\in\argmin_{\pi\in C}\left\{\int_{\R^*} g_\alpha(u(x) - \pi\psi_s(x)) \nu(\ud x) -\pi\varphi_s\right\}$ can be chosen due to arguments using a measurable selection theorem as in Lemma 6 in \cite{morlais2009utility} and the proof of Theorem 3 in \cite[]{morlais2010new}. 

However, this does not imply that $\pi^*$ has the required square integrability properties to be admissible. To attack this question we first bound the arguments where the minimum is attained.
\begin{Prop}\label{minbound}
 Assume that $\nu(\psi_t>0)>0$ and $\nu(\psi_t<0)>0$. Let $c,C\in (0,\infty)$ be such that $\nu(\psi_t>C)>0$ and $\nu(\psi_t<-c)>0$. Then for any $u\in L^2\cap L^\infty(\nu)$ and \begin{equation*}\pi_t^*\in\argmin_{\pi\in C}\left\{\int_{\R^*} g_\alpha(u(x) - \pi\psi_t(x)) \nu(\ud x) -\pi\varphi_t\right\}\end{equation*} it holds that almost surely
\begin{eqnarray*}
&&-3\frac{\|u\|_\infty}C-2\frac{|\varphi_t|}{\alpha\nu(\psi_t>C)C^2}-\frac{\sqrt{2}}{\sqrt{\alpha\nu(\psi_t>C)}C} \sqrt{|u|_\alpha}\\&&\leq 
\pi_t^* \leq 3\frac{\|u\|_\infty}c+2\frac{|\varphi_t|}{\alpha\nu(\psi_t<-c)c^2}+\frac{\sqrt{2}}{\sqrt{\alpha\nu(\psi_t<-c)}c} \sqrt{|u|_\alpha}.
\end{eqnarray*}

\end{Prop}
\begin{proof}
 We only prove the lower bound 
as the proof of the upper one is completely analogous.

Since $0\in C$ and $\lambda$ is strictly convex it suffices to show that for $\pi$ small enough $\lambda(\pi)$ is strictly bigger than $\lambda(0)=|u|_\alpha$, as then the infimum definitely cannot be at these small enough $\pi$. 

We first observe that $g_\alpha$ is non-negative. Hence,
\begin{equation*}
 \lambda(\pi)\geq \int_{\psi_t>C} g_\alpha(u(x) - \pi\psi_t(x)) \nu(\ud x) -\pi\varphi_t.
\end{equation*}
For $\pi \leq -\frac{\|u\|_\infty}{C}$ we have that $u(x)-\pi\psi_t(x)\geq 0$ and so
\begin{eqnarray*}
 \lambda(\pi) &\geq \int_{\psi_t>C}\frac{\alpha}2(u(x)-\pi\psi_t(x))^2\nu(\ud x)-\pi\varphi_t \geq \int_{\psi_t>C}\frac{\alpha}2(\|u\|_\infty+\pi C)^2\nu(\ud x)-\pi\varphi_t\\
&=\frac{\alpha}2\nu(\psi_t>C)\left(C^2\pi^2+\|u\|_\infty^2+2\|u\|_\infty C\pi\right)-\pi\varphi_t.
\end{eqnarray*}
So the question is whether 
\begin{equation*}
 \frac{\alpha}2\nu(\psi_t>C)C^2\pi^2+({\alpha}\nu(\psi_t>C)\|u\|_\infty C-\varphi_t)\pi+\left(\frac{\alpha}2\nu(\psi_t>C)\|u\|_\infty^2-|u|_\alpha\right)
\end{equation*}
is strictly positive. Since in $\pi$ this is a quadratic function, it is elementary to see that this is either always the case or for all $\pi$ smaller than  
\begin{eqnarray*}
 &\frac{-{\alpha}\nu(\psi_t>C)\|u\|_\infty C+\varphi_t-\sqrt{({\alpha}\nu(\psi_t>C)\|u\|_\infty C-\varphi_t)^2-2\alpha\nu(\psi_t>C)C^2\left(\frac{\alpha}2\nu(\psi_t>C)\|u\|_\infty^2-|u|_\alpha\right)}}{\alpha\nu(\psi_t>C)C^2}\\
&\geq \frac{-\|u\|_\infty}C-\frac{|\varphi_t|}{\alpha\nu(\psi_t>C)C^2}-\frac{\|u\|_\infty}C-\frac{|\varphi_t|}{\alpha\nu(\psi_t>C)C^2}-\frac{\|u\|_\infty}C-\frac{\sqrt{2}}{\sqrt{\alpha\nu(\psi_t>C)}C} \sqrt{|u|_\alpha}\\&=-3\frac{\|u\|_\infty}C-2\frac{|\varphi_t|}{\alpha\nu(\psi_t>C)C^2}-\frac{\sqrt{2}}{\sqrt{\alpha\nu(\psi_t>C)}C} \sqrt{|u|_\alpha}
\end{eqnarray*}
using that $\sqrt{a+b}\leq\sqrt a+ \sqrt b$ for $a,b\geq 0$.
\end{proof}

Now we want to link $|u|_\alpha$ to $|u|_{L^2(\nu)}$. For this we need a suitable version of the second part of Corollary 1 in \cite{morlais2009utility} which the following general result provides.
\begin{Lem}\label{BMOequivalence}
Let $g^{(1)},g^{(2)}:\R\to\R$ be two continuous non-negative functions whose only zero is at the origin, with the property that $\exists \epsilon>0,$ and two constants $c_1, c_2>0$ such that $$c_1g^{(1)}(h)\leq g^{(2)}(h) \leq c_2g^{(1)}(h) \quad\text{for }|h|<\epsilon.$$
For any $H(x)$ which is ($\nu$-a.e.) bounded it holds that $\exists \ K>0$ such that  
\begin{align*}
\frac1K \int_{\R^*} g^{(1)}(H(x))\nu(\ud x)&\leq \int_{\R^*}g^{(2)}(H(x))\nu(\ud x)\\
&\leq K\int_{\R^*}g^{(1)}(H(x))\nu(\ud x).
\end{align*}

The constant $K$ can be taken such that it only depends on the functions $g^{(1)}$ and $g^{(2)}$ and the essential upper bound on $H$. 
\end{Lem}
\begin{proof}
Let $\epsilon$ be as in the hypothesis. We have that
\begin{align*}
c_1  \int_{\{|H(x)|<\epsilon\}\cap \R^*}g^{(1)}(H(x))\nu(\ud x)\ &\leq \int_{\{|H(x)|<\epsilon\}\cap \R^*}g^{(2)}(H(x))\nu(\ud x)\\
&\leq c_2\int_{\{|H(x)|<\epsilon\}\cap \R^*}g^{(1)}(H(x))\nu(\ud x).
\end{align*}
On the other hand, the only zero of the two functions is in zero, and since $H$ is bounded $\exists c_3,c_4>0$ such that
\begin{equation}\label{bounds function - bdd process}
c_3\leq\frac{g^{(2)}(H(x))}{g^{(1)}(H(x))}\leq c_4
\end{equation}
on the set $\{|H(x)|\geq\epsilon\}$.\\
Summing up, we obtain that
\begin{align*}
&(c_1\wedge c_3)\!\! \left( \int_{\{|H(x)|<\epsilon\} \cap \R^*}\!\!\! g^{(1)}(H(x))\nu(\ud x) \! +\! \int_{\{|H(x)|\geq \epsilon\}\cap \R^*}\!\!\!g^{(1)}(H(x))\nu (\ud x)\right) \leq\\
&\int_{\R^*}g^{(2)}(H(x))\nu(\ud x)\leq\\
&\left( c_2\vee c_4\right)\!\!\left(\int_{\{|H(x)|<\epsilon\}\cap \R^*}\!\!\!g^{(1)}(H(x))\nu(\ud x)+\int_{\{|H(x)|\geq \epsilon\}\cap \R^*}\!\!\!g^{(1)}(H(x))\nu(\ud x)\right).
\end{align*}
\end{proof}
 
Since $g_\alpha(x)\sim x^2$ for $x\to 0$, the above lemma implies that for bounded $u$ it holds that $|u|_\alpha/K\leq |u|^2_{L^2(\nu)}\leq K |u|_\alpha$ with $K$ only depending on the bound for $u$.

\begin{Prop}\label{intprop}
 Assume  
\begin{itemize}
 \item there exist $c,C,\delta>0$ with $\nu(\psi_t>C)>\delta$ and $\nu(\psi_t<-c)>\delta$,
\item  $|\psi_s|_{L^2(\nu)}$ is bounded on $\Omega\times[0,T]$,
\item the BSDE \eqref{bsde} has a solution with $U\in L^2(\tilde N_p)$ being $\lambda\otimes P$-a.e. bounded and $Z\in L^2(W)$. 
\end{itemize}
Then there exists a predictable $\pi^*$ such that
\begin{enumerate}
\item $\pi_t^*\in\argmin_{\pi\in C}\left\{\int_{\R^*} g_\alpha(U_t(x) - \pi\psi_t(x)) \nu(\ud x) -\pi\varphi_t\right\}$ $\lambda\otimes P$-a.e.,
\item $\int_0^T|\pi^*_s\varphi_s|ds\in L^2(\Omega,P)$ and $\pi^*\psi\in L^2(\widetilde N_p)$,
\item $U(X^{\pi^*,0,x}-Y_t)$ is  a local martingale. 
\end{enumerate}

\end{Prop}
\begin{proof}
We already know that we can find a predictable $\pi^*$ satisfying (i).

 Proposition \ref{minbound} and Lemma \ref{BMOequivalence} show that under the assumptions there are finite constants $K,K'>0$ such that
\begin{equation}
 |\pi_s^*|^2\leq K+K'|U_s|_{L^2(\nu)}^2.\label{eqpibound}
\end{equation}

As $U\in L^2(\tilde N_p)$ this implies, $E\left(\int_0^T|\pi_s|^2ds\right)<\infty$ and so Lemma \ref{intlem} shows (ii).

Finally, (iii) follows immediately, because by the definition of the generator the finite variation part of $U(X^{\pi^*,0,x}-Y_t)$ vanishes.

\end{proof}

Note that we are not discussing whether the BSDE does admit a unique solution. Later on we will give a concrete example where it turns out that our results can be used to solve an optimal hedging problem without discussing these issues.
The last general result is now to strengthen our conditions further in order to ensure admissibility and optimality.
\begin{Teo}\label{thoptimality}
 Assume all conditions of Proposition \ref{intprop} are satisfied and let $\pi^*$ be as characterised there. Assume additionally $|U_s|_{L^2(\nu)}$ is $\lambda\otimes P$-a.e. bounded.
\begin{enumerate}
\item Then $\pi^*$ is bounded and $\pi^*\in \mathcal{A}$.
\item If furthermore $dH_t:=\alpha Z_t \ud W_t +\int_{\R^*} (e^{-\alpha(\pi^*_t\psi_t(x)-U_t(x))}-1)\tilde N(\ud t, \ud x)$ defines a BMO martingale, then $\pi^*$ is an optimal strategy.  
\item If additionally $ Z$ is bounded (recall that boundedness of $U$ was already assumed in Proposition \ref{intprop}) and $U(X_0^{\pi^*}-Y_0)$ has finite expectation, then the process $dH_t:=\allowbreak\alpha Z_t \ud W_t \allowbreak+\int_{\R^*} (e^{-\alpha(\pi^*_t\psi_t(x)-U_t(x))}-1)\tilde N(\ud t, \ud x)$ is a BMO martingale and $\pi^*$ is an optimal strategy.
\end{enumerate}
\end{Teo}
\begin{proof}
 The boundedness of $\pi^*$ follows immediately from \eqref{eqpibound}. This means that we actually can view the optimization problem on our non-compact set $C$ as one on a compact constraint set where the admissibility has been shown in Lemma 1 of \cite{morlais2009utility} and Lemma 2 of \cite{morlais2010new}. 

From the choice of the generator we have that $U(X^{\pi^*}-Y)=U(X_0^{\pi^*}-Y_0)\mathcal{E}(H)$ and so by Kazamaki's criterion (see exercises after Chapter X in \cite{he1992semimartingale} or \cite{kazamaki1979}) $U(X^{\pi^*}-Y)$ is a uniformly integrable martingale. Hence, the martingale optimality principle concludes.

For (iii) it is straightforward to see that $H$ is a local martingale and it has bounded jumps. 
Moreover, for any $t\in [0,T]$ we have \begin{align*}[H,H]_T-[H,H]_t= &\int_t^T Z^2_sds+ \int_t^T\int_{\R^*} (e^{-\alpha(\pi^*_s\psi_s(x)-U_s(x))}-1)^2N(\ud s, \ud x)\\\leq &\int_t^T Z^2_sds +k\int_t^T\int_{\R^*} (\alpha(\pi^*_s\psi_s(x)-U_s(x)))^2N(\ud s, \ud x)\end{align*} with a constant $k>0$ using  an obvious variant of Lemma \ref{BMOequivalence}. Applying the compensation formula for conditional expectations (see e.g. \cite{kyprianou2006introductory}, Corollary 4.5) we get 
\begin{align*}
E([H,H]_T-[H,H]_t|\mathcal{F}_t)\leq& E\left(\left.\int_t^T Z^2_sds+k\int_t^T\int_{\R^*} (\alpha(\pi^*_s\psi_s(x)-U_s(x)))^2 \nu(\ud x)\ud s\right|\mathcal{F}_t\right)\\
\leq & KT+\tilde K E\left(\left.\int_t^T |\psi_s|_{L^2(\nu)}^2+|U_s|_{L^2(\nu)}^2ds\right|\mathcal{F}_t\right)\leq \bar K T
\end{align*}
with $K, \tilde K,\bar K>0$ being appropriate constants, because $Z,\pi, |\psi_s|_{L^2(\nu)}$ and $|U_s|_{L^2(\nu)}$ are all bounded.
 So \cite{he1992semimartingale}, Theorem 10.9, concludes and establishes that $H$ is BMO noting that we are only looking at the finite time horizon $[0,T]$. 
\end{proof}
\begin{oss}
In any exponential pure-jump L\'evy market model with both negative and positive as well as bounded jumps the first two conditions of Proposition \ref{intprop} are always satisfied. So for such models and thus in particular for ``jump-truncated versions'' of the most popular models like NIG, variance Gamma or CGMY the ``only'' question is whether the BSDE has a solution with the demanded properties. 

Regarding the need for the boundedness of the jumps we would conjecture that one could at many instances generalize our results by assuming only the finiteness of suitable (exponential) moments. However, at the moment we have no idea how to prove optimality of $\pi^*$ without using BMO arguments which  absolutely need bounded jumps.  
\end{oss}

\section{Cross-hedging in a jump market model}\label{sec:Cross hedging}
In the same market model as in the previous section, consider an additional illiquid asset with price process \xproc{I} and assume that we want to hedge a position $B=h(I_T)$ at the terminal time, for a payoff function $h:\R\to\R$ to be specified later.\\
We assume that the stock is given as in the previous sections whereas the logspread \xproc{\Xi_t=\log{S_t}-\log I} has dynamics given in terms of the one dimensional independent Brownian motion $W$ and has continuous paths. This is the simplest case. The motivation from applications is that the big moves, the jumps, occur equally in the liquid and illiquid asset and due to small fluctuations in supply and demand of the illiquid asset the price of the illiquid asset fluctuates a bit around a fixed spread (which for our crude oil/jet fuel example below would essentially be given by the costs of the refining process).  However, as we will see the log-spread can also have jumps again.\\
This setting can be e.g. suitable for a company that wants to hedge itself against losses for a massive rise in the price of a needed commodity which is not liquidly traded. A classical example is the one of ``fuel hedging'', where a company (an airline) needs to buy jet fuel on a regular basis. If no futures on the needed commodity are available in the market, the company could think of buying futures in another commodity whose price is strongly correlated with the price of the needed one (in the example above, the company could decide to buy futures on crude oil, which is needed in the production of fuel). We assume, in the simplest case, that the logspread of the two prices evolves according to a diffusion, but has a mean reverting behaviour. See, for example, \cite{AnkirchnerImkeller2011} or \cite{Ankirchneretal2012} and the references therein for a comprehensive introduction to such cross-hedging problems.
\\
It is well-known (cf. e.g. \cite{kallsen2002cumulant}) that the stochastic exponential can always be represented as a classical exponential of another process, denoted here  by $N := \log S$. With this notation, which is more convenient in the following, we have that $I = \exp(N-\Xi)$.

\subsection{Explicit solutions of FBSDEs with affine forward dynamics}
Take the couple $R=(N,\Xi)$ as source of uncertainty.\\
The discounted logarithmic price process will have dynamics \systeq{\ud N_t}{\beta\ud t + \int_{\R^*}\gamma(x)\tilde N_p(\ud t, \ud x)}{N_0}{n} where $\gamma := \log(\psi + 1)$ and $\beta := \varphi-\int_{\R^*}(e^{\gamma(x)}-1-\gamma(x))\nu(\ud x)$. Assume that $\beta, \gamma$ are deterministic and time-independent, and that $\varphi, \psi$  (as implicitly defined by $\beta, \gamma$) satisfy the hypotheses of Theorem \ref{nflvr}. Note that, as announced, we have now changed the parametrisation of our model looking at the dynamics of the logarithmic prices instead of the stochastic logarithm of the prices, as this is more convenient for writing down the model. The solutions of the BSDEs will, however, later on be partially expressed using both parametrizations in order to state everything conveniently.

In the simplest case we assume the logspread to be given by a (mean-reverting) Gaussian Ornstein-Uhlenbeck process with mean zero, i.e. \systeq{\label{spread}\ud \Xi_t}{-B\Xi_t\ud t + \Sigma \ud W_t}{\Xi_0}{\xi,}
where $B,\Sigma$ are real constants, $\Sigma>0$.

However, our approach below works just as well for general Ornstein-Uhlenbeck type dynamics, i.e. 
 \systeq{\label{genspread}\ud \Xi_t}{(b-B\Xi_t)\ud t + \Sigma \ud W_t+\int_{\R^*}\gamma_{\Xi}(x)\tilde N_p(dx,dt)}{\Xi_0}{\xi,}
with an additional constant $b\in\R$ and bounded $\gamma_\Xi\in L^2(\nu)$. These dynamics now have a general mean reversion level $b$ and allow the logspread to have jumps coming from the given Poisson point process. By appropriate choices of $\gamma$ and $\gamma_\Xi$ one can allow the liquid asset price process and the logspread to have both common and individual jumps. 
\\
The assumption on the dynamics of the logspread gives the process $R$ an important property: Since the first component is a L\'evy process, and the second is an Ornstein-Uhlenbeck process, the risk source driving the forward equation of the considered FBSDE turns out to be an affine process.
This family of processes is particularly tractable. In fact, there exist some results on explicit solutions to FBSDEs, where the
forward process is affine, see \cite{richter2012explicit}. \\
In this case, the wealth process has dynamics
\begin{align*}
\ud X^{x,\pi}_t= &\pi_t\underbrace{\left(\beta + \int_{\R^*}(e^{\gamma(x)}-1-\gamma(x))\nu(\ud x)\right)}_{\varphi}\ud t +\pi_t \int_{\R^*}\underbrace{(e^{\gamma(x)}-1)}_{\psi(x)}\tilde N_p(\ud x, \ud t).
\end{align*}
We will consider a system of a forward and a backward stochastic differential equations (FBSDE) where the forward process is given by the $\R^2-$valued process 
\begin{equation}
R_t=\left(\begin{array}{c}N_t\\\Xi_t\end{array}\right).
\end{equation}
with
\begin{align}
\nonumber\ud R_t=&\overbrace{\left(\begin{array}{c}\beta\\b\end{array}\right)}^{\underline\beta}\ud t + \overbrace{\left(\begin{array}{cc}0&0\\0&-B\end{array}\right)}^{\underline B}\left(\begin{array}{c}N_t\\\Xi_t\end{array}\right)\ud t \\
&\nonumber+ \int_{\R^*}\underbrace{\left(\begin{array}{cc}\gamma(x)\\\gamma_\Xi(x)\end{array}\right)}_{\underline \gamma} \tilde N_p(\ud t, \ud x) + \underbrace{\left(\begin{array}{c}0\\ \Sigma\end{array}\right)}_{\underline\Sigma}\ud W_t
\end{align}

The generator, as in the previous section, is obtained by imposing the supermartingale condition on the utility of $X_t-Y_t$,
where for $Y_t$, due to the martingale representation property of the
filtration, we assume the dynamics
\systeq{\label{JBSDE}-\ud Y_t}{f(t,Y_{t-}, Z_t,U_t)\ud t - Z_t\ud W_t-\int_{\R^*}U_{t}(x) \tilde N_p(\ud t, \ud x)}{Y_T}{h(\exp(N_T-\Xi_T)).}
As shown in Section \ref{sec:Model}, the generator has to be taken according to \eqref{generator}. $Y_T$ is, of course, the ``derivative'' of the illiquid asset which we want to hedge. \\
We will first assume that the infimum in \eqref{generator} is attained at some optimal strategy $\pi^*\in \mathcal A$.
Then 
\begin{equation}\label{jumpgen}
f(t,Z_t,U_t) =\int_{\R^*}g_\alpha\left(U_t(x) - \pi_t^*\psi(x)\right)\nu(\ud x)
- \pi^*_t\varphi + \alpha \frac{|Z_t|^2}{2},
\end{equation}
where $\varphi = \beta + \int_{\R^*}(e^{\gamma(x)}-1-\gamma(x))\nu(\ud x)$, $\psi := \exp(\gamma) - 1$ and $(Y,U,Z)$ is a solution of the BSDE with generator $f$ and terminal condition $B = h(I_T) = F(R_T)$, for $F(n,s):=h(\exp(n-s))$.\\
We will see that the above assumption does not lead to cyclic arguments.\\
The idea is now to use an approach like the one in \cite{richter2012explicit} for processes in the positive semidefinite matrices to find an
explicit solution for the studied FBSDE making an affine ansatz. Since we are considering here a somewhat different generator than \cite{richter2012explicit}, we will  give a direct proof.
\begin{Prop}\label{expl_jumps}
Let $\Gamma: [0,T]\times\R^2\to \R^2$ and $\omega(\cdot, a, v):[0,T]\to\R$ be solutions to the following ODEs %generalized Riccati equations
\begin{align}
&\label{ode1}\left\{\begin{array}{ll}
-\frac{\partial\Gamma}{\partial t}(t,a) = &\underline{B}^* \Gamma(t,a),\quad \forall t\in[0,T)\\
\Gamma(T,a) =&a
\end{array}\right.\\
&\label{ode2}\left\{\begin{array}{ll}
-\frac{\partial \omega}{\partial t}(t,a,v)=&\langle\Gamma(t,a),\underline{\beta}\rangle - \pi^*_t\varphi + \int_{\R^*}g_\alpha\left(\langle\Gamma(t,a),\underline{\gamma}(x)\rangle-\pi^*_t\psi(x)\right)\nu(\ud x) \\
&+\alpha \frac{\langle\underline{\Sigma},\Gamma(t,a)\rangle^2}{2}\quad \forall t\in[0,T)\\
\omega(T,a,v)=&v,
\end{array}\right.
\end{align}
for $a\in\R^2$ and $v\in\R$ some constant initial conditions. Here $\underline B^*$ denotes the adjoint operator of $\underline B$.\\
Then, for every $(t,a,v)\in [0,T]\times\R^2\times\R$,
\begin{equation}\label{expl_sol_1}
\left\{\begin{array}{ll}
Y_t&=\left\langle\Gamma(t,a),\(\begin{array}{c}N_t\\ \Xi_t\end{array}\)\right\rangle+\omega(t,a,v)\\
Z_t &=\langle \Gamma(t,a),\underline{\Sigma}\rangle\\
U_t(x)&=\langle\Gamma(t,a), \underline{\gamma}(x)\rangle
\end{array}\right.
\end{equation}
solves the BSDE \eqref{JBSDE} with terminal
condition $$F(N_T,\Xi_T)$$ and generator $$f(t,z,u)= \int_{\R^*}g_\alpha(u(\cdot) - \pi^*\psi(\cdot))\ud \nu - \pi^*\varphi+\alpha \frac{|z|^2}2,$$ where
\begin{equation*}
F(n,s) = \left\langle a,\left(\begin{array}{c}n\\s\end{array}\right)\right\rangle + v.
\end{equation*}
\end{Prop}
Observe that \eqref{ode1} is a linear ODE and \eqref{ode2} just means an integration.
\begin{proof}
Recall that $R_t = (N_t,\Xi_t)$ satisfies 
\begin{equation}
\ud R_t = \underline{\beta} \ud t + \underline{B}R_t \ud t + \underline{\Sigma} \ud W_t + \int_{\R^*} \underline{\gamma}(x) \tilde N_p(\ud t, \ud x)
\end{equation}
and that the generator is given as in \eqref{jumpgen}.\\
The ansatz $Y_t = \langle\Gamma(t,a),(N_t, \Xi_t)^T\rangle+\omega(t,a,v)$ together with an application of It\^o's (backwards) Formula, yields
\begin{align*}
Y_t = &\langle\Gamma(t,a),R_t\rangle + \omega(t,a,v)\stackrel{\text{It\^o}}{=}\\
= &\langle\Gamma(T,a), R_T\rangle + \omega(T,a,v) \\
&-\int_t^T\(\langle\Gamma(s,a),\underline{\beta}\rangle
 + \langle\Gamma(s,a), \underline{B}R_{s-}\rangle \)\ud s- \int_t^T\int_{\R^*} \overbrace{\langle\Gamma(s,a),\underline{\gamma}(x)\rangle}^{U_s(x)}\tilde N_p(\ud s, \ud x)  \\&- \int_t^T\overbrace{\langle\Gamma(s,a),\underline{\Sigma}\rangle}^{Z_s}\ud W_s-
\int_t^T\left(\langle\frac{\partial \Gamma(s,a)}{\partial s},R_{s-}\rangle + \frac{\partial \omega(s,a,v)}{\partial s}\right)\ud s =\\
=& \overbrace{\langle a,R_T\rangle + v}^{F(R_T)} \\
&-\int_t^T\Big(\langle\Gamma(s,a),\underline{\beta}\rangle + \langle\underline{B}^*\Gamma(s,a),R_{s-}\rangle - \langle \underline{B}^*\Gamma(s,a),R_{s-}\rangle -\langle\Gamma(s,a),\underline{\beta}\rangle\\
&+\pi^*_s\varphi
-\int_{\R^*}g_\alpha\left(U_s(x)-\pi^*_s\psi(x)\right)\nu(\ud x)\Big) \ud s - \int_t^T\alpha\frac{|Z_s|^2}{2}\ud s\\
& - \int_t^T\int_{\R^*} U_s(x)\tilde N_p(\ud s,\ud x) - \int_t^T Z_s\ud W_s\\
=&F(R_T) +\int_t^T f(s,Z_s,U_s)\ud s - \int_t^T\int_{\R^*} U_s(x)\tilde N_p(\ud s,\ud x) - \int_t^T Z_s\ud W_s.
\end{align*}
This shows that \eqref{expl_sol_1} gives indeed a solution to the studied FBSDE.
\end{proof}

Thanks to this result, under the assumptions of Theorem \ref{nflvr} we can now compute the optimal strategy by pointwise minimization of the generator in \eqref{jumpgen}. Since the pair $(Z,U)$ above does not depend on $\pi^*$, we can plug it into Equation \eqref{jumpgen}, obtaining
\begin{equation*}
f(t,Z_t,U_t) =\inf_{\pi\in C}\left\{\int_{\R^*}g_\alpha\left(\langle\Gamma(t,a), \underline{\gamma}(x)\rangle - \pi\psi(x)\right)\nu(\ud x) - \pi\varphi \right\} + \alpha \frac{\left|\langle \Gamma(t,a),\underline{\Sigma}\rangle\right|^2}{2}.
\end{equation*}
Computing the infimum we obtain an optimal solution to the utility maximization problem.

\begin{Teo}
 Let $\Gamma$ be as in Proposition \ref{expl_jumps} and assume that $\nu(\psi>0)>0$ and $\nu(\psi<0)>0$. Then any predictable $\pi^*$ satisfying
\begin{equation*}
\pi^*_t\in\argmin_{\pi\in C}\left\{\int_{\R^*}g_\alpha\left(\langle\Gamma(t,a), \underline{\gamma}(x)\rangle - \pi\psi(x)\right)\nu(\ud x) - \pi\varphi \right\}
\end{equation*}
is an optimal cross hedge for the terminal payoff   $\left\langle a,\left(\begin{array}{c}N_T\\\Xi_T\end{array}\right)\right\rangle + v$ for the exponential utility with parameter $\alpha$.
\end{Teo}
Note that the optimal strategy is bounded and can be taken deterministic.

\begin{proof}
 This follows from Theorem \ref{thoptimality} for which all assumptions are straightforward to check using the explicit form of $U, Y$ and $Z$ given above.
\end{proof}

\begin{oss}
(i) Although in this example we are able to achieve explicit solutions, the scope of application is still limited, since having an affine terminal condition in the BSDE means hedging a  logarithmic claim on the illiquid asset. One could think of extending the ``ansatz''-approach to more general terminal conditions, but with the generator associated to our cross-hedging problem we have not succeeded so far. However, for some simpler generators one can get explicit results with more suitable terminal conditions, as  illustrated in the next section.

(ii) Despite what we just said about the relevance of logarithmic (in the price of the illiquid asset) terminal conditions, there is a substantial literature on such ``log-contracts'', as they are equivalent to variance swaps in a Brownian market model and also are related to variance swaps in models with jumps (see \cite{CarrLeeWu2012,CarrLee2013}, for instance). Therefore, an interesting question for future research may be whether our results can be used to hedge positions in the variance/realized volatility of the illiquid asset.

(iii) A close inspection of the previous arguments shows that the results can be extended to $\beta, b, B,\Sigma, \gamma,\gamma_\Xi$ being not constant, but appropriate predictable processes. However, we refrain from carrying out this in detail, as it gives a long list  of technical (boundedness) conditions.
\end{oss}

\section{An exponential ansatz}\label{sec:exp terminal}
One could think of generalizing the approach of Section \ref{sec:Cross hedging} by using a different ansatz for the solution of the BSDE, but this seems to require a restriction on the generator and the terminal condition incompatible with our utility optimization problem. In the following, we will give another example where explicit solutions can be achieved for a quite general generator, which, however, does not allow an infimum  in it as needed for the utility optimization problem. Despite of this, it seems to us an interesting example of a L\'evy-driven BSDE which can be solved explicitly.\\
\begin{Prop}
Consider an FBSDE of the following form:
\begin{equation*}
\left\{\begin{array}{rl}
\ud R_t = &\underline\beta \ud t + \underline BR_t \ud t + \underline\Sigma \ud W_t + \int_{\R^*} \underline\gamma(\xi) \tilde N_p(\ud t, \ud \xi),\\
R_0 =& r,
\end{array}\right.
\end{equation*}
\begin{equation*}
\left\{\begin{array}{rl}
-\ud Y_t = &f(t,Y_t,Z_t,U_t)\ud t - Z_t\ud W_t - \int_{\R^*}U_t(x)\tilde N_p(\ud t,\ud x)\\
Y_T =& F(R_T).
\end{array}\right.
\end{equation*}
Assume now that:
\begin{itemize}
\item the terminal condition is exponential of the form $F(r) = \exp(\langle a,r\rangle)w + v$ for $a\in\R^2,\ v,w\in\R$ constants;
\item the generator $f$ is of the form
\begin{equation}\label{lingen} 
f(s,y,z,u) = c_y(s)y + c_z(s)z + \int_{\R^*}c_u(s)u(x)\nu(\ud x) + c(s)
\end{equation}
with $c_y,c,c_z,c_u:[0,T]\to\R$, continuous functions of time.
\end{itemize}
Let $\Gamma(\cdot, a):[0,T]\to\R^2,\ \omega(\cdot, a, w):[0,T]\to\R$ and $\xi(\cdot, a,w,v):[0,T]\to\R$ be the unique solutions to the following differential equations:
\begin{align*}
&\left\{
\begin{array}{rl}
-\frac{\partial \Gamma}{\partial s}(s,a)=&\underline B^*\Gamma(s,a)\\
\Gamma(T,a)=&a
\end{array}
\right.\\
&\left\{
\begin{array}{rl}
-\frac{\partial \omega}{\partial s}(s,a,w)=&\omega(t,a,w)\[\frac 12\tr\left(\underline\Sigma\underline\Sigma^T\Gamma(s,a)\Gamma(s,a)^T\right) + \langle\Gamma(s,a),\underline\beta\rangle \right.\\
&+ \int_{\R^*}\left(e^{\langle\Gamma(s,a),\underline\gamma(x)\rangle}-1-\langle\Gamma(s,a),\underline\gamma(x)\rangle\right)\nu(\ud x) \\ &\left.+c_y(s) +c_z(s) \langle\Gamma(s,a),\underline\Sigma\rangle +c_u(s)\int_{\R^*}\(e^{\langle\Gamma(s,a),\underline\gamma(\cdot)\rangle}-1\)\nu(\ud x)\]\\
\omega(T,a,w)=&w
\end{array}
\right.\\
&\left\{
\begin{array}{rl}
-\frac{\partial \xi}{\partial s}(s,a,w,v)=&c_y(s)\xi(s,a,w,v) +c(s)\\
\xi(T,a,w,v)=&v. 
\end{array}
\right.
\end{align*}
where $\underline B^*$ is the adjoint operator of $\underline B$.\\
Then the triplet of adapted/predictable processes
\begin{equation}\label{exp_solution}
\left\{\begin{array}{rl}
Y_t&=\exp\left(\langle\Gamma(t,a),R_t\rangle\right)\omega(t,a,w) + \xi(t,a,w,v),\\
Z_t &=\exp\left(\langle\Gamma(t,a),R_{t-}\rangle \right)\omega(t,a,w)\langle\Gamma(t,a),\underline\Sigma\rangle \\
U_t(x)&=\exp\left(\langle\Gamma(t,a),R_{t-}\rangle\right)\omega(t,a,w) \left(e^{\langle\Gamma(t,a),\underline\gamma(x)\rangle}-1\right),
\end{array}\right.
\end{equation}
solves the FBSDE.
\end{Prop}
We do not want to discuss $Y,Z,U$ in any detail, but note that, since $R$ has exponential moments of all orders and the ODEs are at most linear, it is obvious that they satisfy square integrability conditions ensuring all relevant stochastic integrals to be well-defined.
\begin{proof}
Apply It\^o's Formula to the regular function $h(t,x):=\exp(\langle\Gamma(t,a),x\rangle )\omega(t,a,w) + \xi(t,a,w,v)$. The ansatz $Y_t=h(t,R_t)=\exp\left(\langle\Gamma(t,a),R_t\rangle\right)\omega(t,a,w) + \xi(t,a,w,v)$ together with the assumptions on the time evolution of the coefficients $\Gamma, \omega$ and $\xi$ yields
\begin{align*}
Y_t = &\exp\left(\langle\Gamma(T,a),R_T\rangle \right)\omega(T,a,w) + \xi(T,a,w,v) \\ 
&- \int_t^T\left\{\exp\left(\langle\Gamma(s,a),R_{s-}\rangle \right)\left[\omega(s,a,w)\left(\left\langle\frac{\partial \Gamma(s,a)}{\partial s},R_{s-}\right\rangle +  \frac 12\tr\left( \underline\Sigma\underline\Sigma^T\Gamma(s,a) \Gamma(s,a)^T\right)\right) \right.\right.\\
&+\left.\left.  \frac{\partial \omega(s,a,w)}{\partial s}\right] + \frac{\partial \xi(s,a,w,v)}{\partial s}\right\}\ud s\\
&-\int_t^T\exp\left(\langle\Gamma(s,a),R_{s-}\rangle\right)\omega(s,a,w)\langle\Gamma(s,a),\underline\beta\ud s + \underline BR_{s-}\ud s + \underline \Sigma \ud W_s\rangle\\
&-\int_t^Te^{\langle\Gamma(s,a),R_{s-}\rangle}\omega(s,a,w)\int_{\R^*}\left(\exp\left(\langle\Gamma(s,a),\underline\gamma(x)\rangle\right)-1\right)\tilde N_p(\ud s, \ud x)\\
&-\int_t^Te^{\langle\Gamma(s,a),R_{s-}\rangle}\omega(s,a,w) \int_{\R^*}\left(\exp\left(\langle\Gamma(s,a),\underline\gamma(x)\rangle\right)-1-\langle\Gamma(s,a),\underline\gamma(x)\rangle\right)\nu(\ud x) \ud s
\\
=&F(N_T,\Xi_T) \\ 
&-\int_t^T e^{\langle\Gamma(s,a),R_{s-}\rangle}\left\{\omega(s,a,w) \left[ \left\langle\frac{\partial \Gamma(s,a)}{\partial s},R_{s-}\right\rangle + \frac 12\tr\left(\underline\Sigma\underline\Sigma^T\Gamma(s,a)\Gamma(s,a)^T\right) \right.\right.\\
&\left. + \langle\Gamma(s,a),\underline\beta\rangle-\langle\Gamma(s,a),\underline BR_{s-}\rangle+\int_{\R^*}\left(e^{\Gamma(s,a)\underline\gamma(x)}-1-\langle\Gamma(s,a),\underline\gamma(x)\rangle\right)\nu(\ud x)\right]\\
&\left.+ \frac{\partial \omega(s,a,w)}{\partial s} \right\}\ud s - \int_t^T \frac{\partial}{\partial s}\xi(s,a,w,v) \ud s\\
&-\int_t^T\underbrace{\exp\left(\langle\Gamma(s,a),R_{s-}\rangle \right) \omega(s,a,w) \langle\Gamma(s,a),\underline\Sigma\rangle}_{:=Z_s}\ud  W_s\\
&-\int_t^T\int_{\R^*} \underbrace{\exp\left(\langle\Gamma(s,a),R_{s-}\rangle\right)\omega(s,a,w)\left(e^{\langle\Gamma(s,a),\underline\gamma(x)\rangle}-1\right)}_{:=U_s(x)}\tilde N_p(\ud s, \ud x)=\\
=&F(N_T,\Xi_T) \\ 
&-\int_t^T \Big\{e^{\langle\Gamma(s,a),R_{s-}\rangle}\omega(s,a,w) \left[ -c_y(s) -c_z(s)\langle \Gamma(s,a),\underline \Sigma\rangle -c_u(s)\int_{\R^*}\(e^{\langle\Gamma(s,a),\underline\gamma\rangle}-1\)\nu(\ud x)\right]\\
&-c_y(s)\xi(s,a,w,v) -c(s) \Big\}\ud s\\
&- \int_t^T\int_{\R^*} U_s(x)\tilde N_p(\ud s,\ud x) - \int_t^T Z_s\ud W_s=\\
=&F(N_T,\Xi_T) \\ 
&-\int_t^T \Big\{-c_y(s) \underbrace{\(e^{\langle\Gamma(s,a),R_{s-}\rangle}\omega(s,a,w) +\xi(s,a,w,v)\)}_{Y_{s-}} -c_z(s)\underbrace{e^{\langle\Gamma(s,a),R_{s-}\rangle}\omega(s,a,w)\langle \Gamma(s,a),\underline \Sigma\rangle}_{Z_s} \\
&-c_u(s)\int_{\R^*}\underbrace{e^{\langle\Gamma(s,a),R_{s-}\rangle}\omega(s,a,w)\(e^{\langle\Gamma(s,a),\underline\gamma(\cdot)\rangle}-1\)}_{U_s(x)}\nu(\ud x)-c(s) \Big\}\ud s\\
&- \int_t^T\int_{\R^*} U_s(x)\tilde N_p(\ud s,\ud x) - \int_t^T Z_s\ud W_s=\\
=&F(N_T,\Xi_T) +\int_t^T f(s,Y_{s-},Z_s,U_s)\ud s - \int_t^T\int_{\R^*} U_s(x)\tilde N_p(\ud s,\ud x) - \int_t^T Z_s\ud W_s.
\end{align*} 
This proves that the triplet given in \eqref{exp_solution} solves the studied FBSDE with terminal condition  $F(r) = \exp(\langle a,r\rangle)w + v$ and generator as in \eqref{lingen}.
\end{proof}
\begin{oss}
Observing the structure of the controls, one can see that this approach results in determining some ODE just in the case when the generator is linear affine in $Y$, $Z$ and $U$. This is due to the fact that no random factors can appear in the time evolution of the functions parametrizing the ansatz.
\end{oss}
\section*{Acknowlegments}
C.M. gratefully acknowledges support by the DFG Graduiertenkolleg 1100.

%\bibliographystyle{apalike}
%\bibliography{bibliography}
\end{document}